\documentclass{nyjm}

\usepackage{enumitem,mathrsfs,mleftright}
\usepackage[colorlinks]{hyperref}
\usepackage[capitalize,nameinlink]{cleveref}

\newtheorem{Thm}{Theorem}[section]
\newtheorem{Prop}[Thm]{Proposition}

\theoremstyle{definition}
\newtheorem{Def}[Thm]{Definition}
\newtheorem{Rmk}[Thm]{Remark}

\setlist[itemize]{leftmargin = *}
\setlist[enumerate]{leftmargin = *}

\allowdisplaybreaks

\newcommand{\B}[1]{\func{\mathscr{B}}{#1}}
\newcommand{\C}{\mathbb{C}}
\newcommand{\D}{\displaystyle}
\newcommand{\K}{\mathbb{K}}
\newcommand{\N}{\mathbb{N}}
\newcommand{\R}{\mathbb{R}}
\newcommand{\Br}[1]{\( #1 \)}
\newcommand{\df}{\stackrel{\operatorname{df}}{=}}
\newcommand{\OC}[2]{\( #1,#2 \]}
\newcommand{\OO}[2]{\Br{#1,#2}}
\newcommand{\RV}[2]{\func{\mathsf{RV}}{#1,#2}}
\newcommand{\Abs}[1]{\mleft| #1 \mright|}
\newcommand{\Int}[3]{\int_{#1} #2 ~ \mathrm{d}{#3}}
\newcommand{\Seq}[2]{\Br{#1}_{#2}}
\newcommand{\Set}[2]{\mleft\{ #1 ~ \middle| ~ #2 \mright\}}
\newcommand{\Card}[1]{\func{\mathsf{Card}}{#1}}
\newcommand{\func}[2]{#1 \Br{#2}}
\newcommand{\Pair}[2]{\Br{#1,#2}}
\newcommand{\SqBr}[1]{\[ #1 \]}
\newcommand{\SSet}[1]{\mleft\{ #1 \mright\}}
\newcommand{\Trip}[3]{\Br{#1,#2,#3}}
\newcommand{\Power}[1]{\func{\mathcal{P}}{#1}}
\newcommand{\converges}[1]{\stackrel{#1}{\longrightarrow}}

\renewcommand{\(}{\mleft(}
\renewcommand{\)}{\mright)}
\renewcommand{\[}{\mleft[}
\renewcommand{\]}{\mright]}
\renewcommand{\c}{\mathsf{c}}
\renewcommand{\P}{\mathsf{P}}

\begin{document}



\title[Generalization of a real-analysis result]{Generalization of a real-analysis result to \\ a class of topological vector spaces}
\author{Leonard T. Huang}
\email{Leonard.Huang@Colorado.EDU}
\address{Department of Mathematics \\ University of Colorado at Boulder \\ Campus Box 395 \\ Boulder \\ Colorado 80309 \\ United States of America}

\maketitle



\begin{abstract}
In this paper, we generalize an elementary real-analysis result to a class of topological vector spaces. We also give an example of a topological vector space to which the result cannot be generalized.
\end{abstract}



\section{Introduction}


This paper draws its inspiration from the following result, which appears to be a popular real-analysis exam problem (see \cite{KU}, for example):

Let $ \Seq{x_{n}}{n \in \N} $ be a sequence in $ \R $. If $ \D \lim_{n \to \infty} \Br{2 x_{n + 1} - x_{n}} = x $ for some $ x \in \R $, then $ \D \lim_{n \to \infty} x_{n} = x $.

A quick proof can be given using the Stolz-Ces\`aro Theorem.

A natural question to ask is: Is this result still valid if $ \R $ is replaced by another topological vector space? The answer happens to be affirmative for a wide class of topological vector spaces that includes all the locally convex ones.

We will also exhibit a topological vector space for which the result is not valid, which indicates that it is rather badly behaved.

In this paper, we adopt the following conventions:
\begin{itemize}
\item
$ \N $ denotes the set of all positive integers, and for each $ n \in \N $, let $ \SqBr{n} \df \N_{\leq n} $.

\item
All vector spaces are over the field $ \K \in \SSet{\R,\C} $.
\end{itemize}



\section{Good topological vector spaces}


Recall that a topological vector space is an ordered pair $ \Pair{V}{\tau} $, where:
\begin{itemize}
\item
$ V $ is a vector space, and

\item
$ \tau $ is a topology on $ V $, under which vector addition and scalar multiplication are continuous operations.
\end{itemize}



\begin{Def}
Let $ \Pair{V}{\tau} $ be a topological vector space, and $ \Seq{x_{\lambda}}{\lambda \in \Lambda} $ a net in $ V $. Then $ x \in V $ is called a \emph{$ \tau $-limit} for $ \Seq{x_{\lambda}}{\lambda \in \Lambda} $ --- which we write as $ \Seq{x_{\lambda}}{\lambda \in \Lambda} \converges{\tau} x $ --- if and only if for each $ \tau $-neighborhood $ U $ of $ x $, there is a $ \lambda_{0} \in \Lambda $ such that $ x_{\lambda} \in U $ for all $ \lambda \in \Lambda_{\geq \lambda_{0}} $.
\end{Def}



\begin{Rmk}
We do not assume that $ \tau $ is a Hausdorff topology on $ V $.
\end{Rmk}



\begin{Def}
A topological vector space $ \Pair{V}{\tau} $ is said to be \emph{good} if and only if any sequence $ \Seq{x_{n}}{n \in \N} $ in $ V $ has a $ \tau $-limit whenever $ \Seq{2 x_{n + 1} - x_{n}}{n \in \N} $ has a $ \tau $-limit.

A topological vector space that is not good is said to be \emph{bad}.
\end{Def}



\begin{Prop} \label{Limits Are Topologically Indistinguishable}
Let $ \Pair{V}{\tau} $ be a topological vector space, and $ \Seq{x_{n}}{n \in \N} $ a sequence in $ V $ such that $ \Seq{2 x_{n + 1} - x_{n}}{n \in \N} \converges{\tau} x $ for some $ x \in V $. Then either
\begin{itemize}
\item
$ \Seq{x_{n}}{n \in \N} \converges{\tau} x $ also, or

\item
$ \Seq{x_{n}}{n \in \N} $ has no $ \tau $-limit.
\end{itemize}
\end{Prop}

\begin{proof}
If $ \Seq{x_{n}}{n \in \N} $ has no $ \tau $-limit, then we are done.

Next, suppose that $ \Seq{x_{n}}{n \in \N} \converges{\tau} y $ for some $ y \in V $. Then
$$
                 \Seq{2 x_{n + 1} - x_{n}}{n \in \N}
\converges{\tau} 2 y - y
=                y,
$$
so $ y $ is a $ \tau $-limit for $ \Seq{2 x_{n + 1} - x_{n}}{n \in \N} $ in addition to $ x $. It follows that
$$
                 \Seq{0_{V}}{n \in \N}
=                \Seq{\Br{2 x_{n + 1} - x_{n}} - \Br{2 x_{n + 1} - x_{n}}}{n \in \N}
\converges{\tau} x - y,
$$
which yields
$$
                 \Seq{y}{n \in \N}
=                \Seq{0_{V} + y}{n \in \N}
\converges{\tau} \Br{x - y} + y
=                x.
$$
Therefore, any $ \tau $-neighborhood of $ x $ also contains $ y $, giving us $ \Seq{x_{n}}{n \in \N} \converges{\tau} x $.
\end{proof}


\cref{Limits Are Topologically Indistinguishable} tells us: To prove that a topological vector space $ \Pair{V}{\tau} $ is good, it suffices to prove that for each sequence $ \Seq{x_{n}}{n \in \N} $ in $ V $, if $ \Seq{2 x_{n + 1} - x_{n}}{n \in \N} \converges{\tau} x $ for some $ x \in V $, then $ \Seq{x_{n}}{n \in \N} \converges{\tau} x $ also.



\begin{Def}
Let $ p \in \OC{0}{1} $. A \emph{$ p $-homogeneous seminorm} on a vector space $ V $ is then a function $ \sigma: V \to \R_{\geq 0} $ with the following properties:
\begin{enumerate}
\item
\textbf{The Triangle Inequality:} $ \func{\sigma}{x + y} \leq \func{\sigma}{x} + \func{\sigma}{y} $ for all $ x,y \in V $.

\item
\textbf{$ p $-Homogeneity:} $ \func{\sigma}{k x} = \Abs{k}^{p} \func{\sigma}{x} $ for all $ k \in \K $ and $ x \in V $.
\end{enumerate}
\end{Def}



\begin{Rmk}
\begin{itemize}
\item
By letting $ k = 0 $ and $ x = 0_{V} $ in (2), we find that $ \func{\sigma}{0_{V}} = 0 $.

\item
A $ 1 $-homogeneous seminorm is the same as a seminorm in the ordinary sense.

\item
No extra generality is gained by postulating that $ \func{\sigma}{k x} \leq \Abs{k}^{p} \func{\sigma}{x} $ for all $ k \in \K $ and $ x \in V $. If $ k \in \K \setminus \SSet{0} $, then replacing $ k $ by $ \dfrac{1}{k} $ gives us the reverse inequality, which leads to equality; if $ k = 0 $, then equality automatically holds.

\item
We do not consider $ p \in \OO{2}{\infty} $ because
\begin{align*}
\forall x \in V: \quad
       2^{p} \func{\sigma}{x}
& =    \func{\sigma}{2 x} \qquad \Br{\text{By $ p $-homogeneity.}} \\
& =    \func{\sigma}{x + x} \\
& \leq 2 \func{\sigma}{x}, \qquad \Br{\text{By the Triangle Inequality.}}
\end{align*}
so if $ \sigma $ is non-trivial, then $ 2^{p} \leq 2 $, which implies that $ p \in \OC{0}{1} $ if $ p \in \R_{> 0} $.
\end{itemize}
\end{Rmk}


Let $ V $ be a vector space, and $ \mathcal{S} $ a collection of $ p $-homogeneous seminorms on $ V $ where $ p \in \OC{0}{1} $ may not be fixed. Define a function $ \mathcal{U}: V \times \mathcal{S} \times \R_{> 0} \to \Power{V} $ by
$$
\forall x \in V, ~ \forall \sigma \in \mathcal{S}, ~ \forall \epsilon \in \R_{> 0}: \quad
\mathcal{U}_{x,\sigma,\epsilon} \df \Set{y \in V}{\func{\sigma}{y - x} < \epsilon}.
$$
Then let $ \tau_{\mathcal{S}} $ denote the topology on $ V $ that is generated by the sub-base
$$
\Set{\mathcal{U}_{x,\sigma,\epsilon} \in \Power{V}}{\Trip{x}{\sigma}{\epsilon} \in V \times \mathcal{S} \times \R_{> 0}}.
$$



\begin{Prop} \label{Some Properties of tau_S}
The following statements about $ \tau_{\mathcal{S}} $ hold:
\begin{enumerate}
\item
$ \tau_{\mathcal{S}} $ is a vector-space topology on $ V $.

\item
Let $ \Seq{x_{\lambda}}{\lambda \in \Lambda} $ be a net in $ V $. Then for each $ x \in V $, we have
$$
\Seq{x_{\lambda}}{\lambda \in \Lambda} \converges{\tau_{\mathcal{S}}} x
\qquad \iff \qquad
\lim_{\lambda \in \Lambda} \func{\sigma}{x_{\lambda} - x} = 0 ~ \text{for all} ~ \sigma \in \mathcal{S}.
$$
\end{enumerate}
\end{Prop}

\begin{proof}
One only has to imitate the proof in the case of locally convex topological vector spaces that the initial topology generated by a collection of seminorms is a vector-space topology. We refer the reader to Chapter 1 of \cite{R} for details.
\end{proof}



\begin{Prop} \label{Main Theorem}
$ \Pair{V}{\tau_{\mathcal{S}}} $ is a good topological vector space.
\end{Prop}

\begin{proof}
Let $ \Seq{x_{n}}{n \in \N} $ be a sequence in $ V $. Suppose that $ \Seq{2 x_{n + 1} - x_{n}}{n \in \N} \converges{\tau_{\mathcal{S}}} x $ for some $ x \in V $. Then without loss of generality, we may assume that $ x = 0_{V} $. To see why, define a new sequence $ \Seq{y_{n}}{n \in \N} $ in $ V $ by $ y_{n} \df x_{n} - x $ for all $ n \in \N $, so that
\begin{align*}
\forall n \in \N: \quad
    2 y_{n + 1} - y_{n}
& = 2 \Br{x_{n + 1} - x} - \Br{x_{n} - x} \\
& = 2 x_{n + 1} - 2 x - x_{n} + x \\
& = \Br{2 x_{n + 1} - x_{n}} - x.
\end{align*}
Hence,
$$
                               \Seq{2 y_{n + 1} - y_{n}}{n \in \N}
=                              \Seq{\Br{2 x_{n + 1} - x_{n}} - x}{n \in \N}
\converges{\tau_{\mathcal{S}}} x - x
=                              0_{V},
$$
so if we can prove that $ \Seq{y_{n}}{n \in \N} \converges{\tau_{\mathcal{S}}} 0_{V} $, then $ \Seq{x_{n}}{n \in \N} \converges{\tau_{\mathcal{S}}} x $ as desired.

Let $ \sigma \in \mathcal{S} $ and $ \epsilon > 0 $, and suppose that $ \sigma $ is $ p $-homogeneous for some $ p \in \OC{0}{1} $. Then by (2) of \cref{Some Properties of tau_S}, there is an $ N \in \N $ such that
$$
\forall n \in \N_{\geq N}: \quad
  \func{\sigma}{2 x_{n + 1} - x_{n}}
= \func{\sigma}{\Br{2 x_{n + 1} - x_{n}} - 0_{V}}
< \Br{2^{p} - 1} \epsilon.
$$
By $ p $-homogeneity, we thus have
\begin{align*}
\forall k \in \N: \quad
    \func{\sigma}{2^{k} x_{N + k} - 2^{k - 1} x_{N + k - 1}}
& = \func{\sigma}{2^{k - 1} \Br{2 x_{N + k} - x_{N + k - 1}}} \\
& = 2^{\Br{k - 1} p} \func{\sigma}{2 x_{N + k} - x_{N + k - 1}} \\
& < 2^{\Br{k - 1} p} \Br{2^{p} - 1} \epsilon.
\end{align*}
Next, a telescoping sum in conjunction with the Triangle Inequality yields
\begin{align*}
\forall m \in \N: \quad
       \func{\sigma}{2^{m} x_{N + m} - x_{N}}
& =    \func{\sigma}{\sum_{k = 1}^{m} \Br{2^{k} x_{N + k} - 2^{k - 1} x_{N + k - 1}}} \\
& \leq \sum_{k = 1}^{m} \func{\sigma}{2^{k} x_{N + k} - 2^{k - 1} x_{N + k - 1}} \\
& <    \sum_{k = 1}^{m} 2^{\Br{k - 1} p} \Br{2^{p} - 1} \epsilon \\
& =    \Br{2^{m p} - 1} \epsilon.
\end{align*}
Then by $ p $-homogeneity again,
\begin{align*}
\forall m \in \N: \quad
    \func{\sigma}{x_{N + m} - \frac{1}{2^{m}} x_{N}}
& = \func{\sigma}{\frac{1}{2^{m}} \Br{2^{m} x_{N + m} - x_{N}}} \\
& = \frac{1}{2^{m p}} \func{\sigma}{2^{m} x_{N + m} - x_{N}} \\
& < \Br{1 - \frac{1}{2^{m p}}} \epsilon.
\end{align*}
Applying the Triangle Inequality and $ p $-homogeneity once more, we get
$$
\forall m \in \N: \quad
    \func{\sigma}{x_{N + m}}
< \func{\sigma}{\frac{1}{2^{m}} x_{N}} + \Br{1 - \frac{1}{2^{m p}}} \epsilon
= \frac{1}{2^{m p}} \func{\sigma}{x_{N}} + \Br{1 - \frac{1}{2^{m p}}} \epsilon.
$$
Consequently,
$$
     \limsup_{n \to \infty} \func{\sigma}{x_{n}}
=    \limsup_{m \to \infty} \func{\sigma}{x_{N + m}}
\leq \limsup_{m \to \infty} \SqBr{\frac{1}{2^{m p}} \func{\sigma}{x_{N}} + \Br{1 - \frac{1}{2^{m p}}} \epsilon}
=    \epsilon.
$$
As $ \epsilon > 0 $ is arbitrary, we obtain
$$
  \lim_{n \to \infty} \func{\sigma}{x_{n} - 0_{V}}
= \lim_{n \to \infty} \func{\sigma}{x_{n}}
= 0.
$$
Finally, as $ \sigma \in \mathcal{S} $ is arbitrary, (2) of \cref{Some Properties of tau_S} says that $ \Seq{x_{n}}{n \in \N} \converges{\tau_{\mathcal{S}}} 0_{V} $.
\end{proof}


By \cref{Main Theorem}, the class of good topological vector spaces includes:
\begin{itemize}
\item
All locally convex topological vector spaces.

\item
All $ L^{p} $-spaces for $ p \in \OO{0}{1} $, which are generally not locally convex.
\end{itemize}

In the next section, we will give an example of a bad topological vector space.



\section{A bad topological vector space from probability theory}


Before we present the example, let us first fix some probabilistic terminology.



\begin{Def}
Let $ \Trip{\Omega}{\Sigma}{\P} $ be a probability space.
\begin{itemize}
\item
A measurable function from $ \Pair{\Omega}{\Sigma} $ to $ \Pair{\R}{\B{\R}} $ is called a \emph{random variable}.\footnote{$ \B{\R} $ denotes the Borel $ \sigma $-algebra generated by the standard topology on $ \R $.}

\item
The $ \R $-vector space of random variables on $ \Pair{\Omega}{\Sigma} $ is denoted by $ \RV{\Omega}{\Sigma} $.

\item
Let $ \Seq{X_{\lambda}}{\lambda \in \Lambda} $ be a net in $ \RV{\Omega}{\Sigma} $, and let $ X \in \RV{\Omega}{\Sigma} $. Then $ \Seq{X_{\lambda}}{\lambda \in \Lambda} $ is said to \emph{converge in probability} to $ X $ (for $ \P $) if and only if for each $ \epsilon > 0 $, we have
$$
\lim_{\lambda \in \Lambda} \func{\P}{\Set{\omega \in \Omega}{\Abs{\func{X_{\lambda}}{\omega} - \func{X}{\omega}} > \epsilon}} = 0,
$$
in which case, we write $ \Seq{X_{\lambda}}{\lambda \in \Lambda} \converges{\P} X $.
\end{itemize}
\end{Def}


The following theorem says that convergence in probability is convergence with respect to a vector-space topology on the vector space of random variables.



\begin{Thm} \label{Convergence in Probability Is Convergence With Respect To a Vector-Space Topology on the Vector Space of Random Variables}
Let $ \Trip{\Omega}{\Sigma}{\P} $ be a probability space, and define a pseudo-metric $ \rho_{\P} $ on $ \RV{\Omega}{\Sigma} $ by
$$
\forall X,Y \in \RV{\Omega}{\Sigma}: \quad
\func{\rho_{\P}}{X,Y} \df \Int{\Omega}{\frac{\Abs{X - Y}}{1 + \Abs{X - Y}}}{\P}.
$$
Then the topology $ \tau_{\P} $ on $ \RV{\Omega}{\Sigma} $ generated by $ \rho_{\P} $ has the following properties:
\begin{itemize}
\item
$ \tau_{\P} $ is a vector-space topology.

\item
Let $ \Seq{X_{\lambda}}{\lambda \in \Lambda} $ be a net in $ \RV{\Omega}{\Sigma} $. Then for each $ X \in \RV{\Omega}{\Sigma} $, we have
$$
\Seq{X_{\lambda}}{\lambda \in \Lambda} \converges{\P} X
\qquad \iff \qquad
\Seq{X_{\lambda}}{\lambda \in \Lambda} \converges{\tau_{\P}} X.
$$
\end{itemize}
\end{Thm}

\begin{proof}
Please refer to Problems 6, 10 and 14 in Section 5.2 of \cite{RF}.
\end{proof}


Now, for each $ k \in \N $, define a probability measure $ \c_{k} $ on $ \Pair{\SqBr{k}}{\Power{\SqBr{k}}} $ by
$$
\forall A \subseteq \SqBr{k}: \quad
\func{\c_{k}}{A} \df \frac{\Card{A}}{k},
$$
and let $ \Trip{\Omega}{\Sigma}{\P} $ denote the product probability space $ \D \prod_{k = 1}^{\infty} \Trip{\SqBr{k}}{\Power{\SqBr{k}}}{\c_{k}} $. Define a sequence $ \Seq{S_{n}}{n \in \N} $ in $ \Sigma $ by
$$
\forall n \in \N: \quad
S_{n} \df \Set{\mathbf{v} \in \prod_{k = 1}^{\infty} \SqBr{k}}{\func{\mathbf{v}}{n} = 1}.
$$
Then $ \func{\P}{S_{n}} = \dfrac{1}{n} $ for all $ n \in \N $, and the $ S_{n} $'s form mutually-independent events.

Next, define a sequence $ \Seq{Y_{n}}{n \in \N} $ in $ \RV{\Omega}{\Sigma} $ by
$$
\forall n \in \N: \quad
Y_{n} \df 2^{n} \chi_{S_{n}},
$$
where $ \chi_{S_{n}} $ denotes the indicator function of $ S_{n} $. Then we get for each $ \epsilon > 0 $ that
$$
  \lim_{n \to \infty} \func{\P}{\Set{\omega \in \Omega}{\Abs{\func{Y_{n}}{\omega}} > \epsilon}}
= \lim_{n \to \infty} \func{\P}{S_{n}}
= \lim_{n \to \infty} \frac{1}{n}
= 0.
$$
The first equality is obtained because, for each $ \epsilon > 0 $, we have $ 2^{n} > \epsilon $ for all $ n \in \N $ large enough. Consequently, $ \Seq{Y_{n}}{n \in \N} \converges{\P} 0_{\Omega \to \R} $.

Define a new sequence $ \Seq{X_{n}}{n \in \N} $ in $ \RV{\Omega}{\Sigma} $ by
$$
\forall n \in \N: \quad
X_{n} \df
\begin{cases}
0_{\Omega \to \R}                                 & \text{if} ~ n = 1; \\
\D \sum_{k = 1}^{n - 1} \frac{1}{2^{n - k}} Y_{k} & \text{if} ~ n \geq 2.
\end{cases}
$$
Then $ 2 X_{2} - X_{1} = 2 X_{2} = Y_{1} $, and
\begin{align*}
\forall n \in \N_{\geq 2}: \quad
    2 X_{n + 1} - X_{n}
& = 2 \sum_{k = 1}^{n} \frac{1}{2^{n + 1 - k}} Y_{k} - \sum_{k = 1}^{n - 1} \frac{1}{2^{n - k}} Y_{k} \\
& = \sum_{k = 1}^{n} \frac{1}{2^{n - k}} Y_{k} - \sum_{k = 1}^{n - 1} \frac{1}{2^{n - k}} Y_{k} \\
& = Y_{n}.
\end{align*}
It follows that $ \Seq{2 X_{n + 1} - X_{n}}{n \in \N} = \Seq{Y_{n}}{n \in \N} \converges{\P} 0_{\Omega \to \R} $.

Gathering what we have thus far, observe that
\begin{align*}
\forall n \in \N: \quad
       X_{2 n + 1}
& =    \sum_{k = 1}^{2 n} \frac{1}{2^{2 n + 1 - k}} Y_{k} \\
& =    \sum_{k = 1}^{2 n} \frac{1}{2^{2 n + 1 - k}} \Br{2^{k} \chi_{S_{k}}} \\
& =    \sum_{k = 1}^{2 n} 2^{2 k - 2 n - 1} \chi_{S_{k}} \\
& \geq \sum_{k = n + 1}^{2 n} 2^{2 k - 2 n - 1} \chi_{S_{k}} \\
& \geq \sum_{k = n + 1}^{2 n} \chi_{S_{k}} \\
& \geq \chi_{\bigcup_{k = n + 1}^{2 n} S_{k}}.
\end{align*}
As the $ S_{k} $'s are mutually independent, their complements are as well, so
\begin{align*}
\forall n \in \N: \quad
       \func{\P}{\Set{\omega \in \Omega}{\Abs{\func{X_{2 n + 1}}{\omega}} > \frac{1}{2}}}
& \geq \func{\P}{\bigcup_{k = n + 1}^{2 n} S_{k}} \\
& =    1 - \func{\P}{\Omega \Bigg\backslash \bigcup_{k = n + 1}^{2 n} S_{k}} \\
& =    1 - \func{\P}{\bigcap_{k = n + 1}^{2 n} \Omega \setminus S_{k}} \\
& =    1 - \prod_{k = n + 1}^{2 n} \func{\P}{\Omega \setminus S_{k}} \\
& =    1 - \prod_{k = n + 1}^{2 n} \Br{1 - \frac{1}{k}} \\
& =    1 - \prod_{k = n + 1}^{2 n} \frac{k - 1}{k} \\
& =    1 - \frac{n}{2 n} \\
& =    1 - \frac{1}{2} \\
& =    \frac{1}{2}.
\end{align*}
Hence, $ \Seq{X_{n}}{n \in \N} $ does not converge to $ 0_{\Omega \to \R} $ in probability. By \cref{Convergence in Probability Is Convergence With Respect To a Vector-Space Topology on the Vector Space of Random Variables}:



\begin{Prop}
$ \Pair{\RV{\Omega}{\Sigma}}{\tau_{\P}} $ is therefore a bad topological vector space.
\end{Prop}


By \cref{Limits Are Topologically Indistinguishable}, $ \Seq{X_{n}}{n \in \N} $ does not, in fact, converge in probability at all.



\section{Acknowledgments}


The author would like to express his deepest thanks to Dr. Jochen Wengenroth for communicating his example above of a bad topological vector space.




\end{document}